\documentclass[a4paper,10pt,twoside,reqno]{amsart}

\usepackage[latin1]{inputenc}
\usepackage{amsmath, amsfonts, amssymb,amsthm}
\usepackage[mathscr]{eucal} 
\usepackage[pdftex]{graphicx}
\usepackage[all]{xy}
\usepackage{xcolor}
\usepackage{hyperref}
\usepackage[T1]{fontenc}
\usepackage[sc]{mathpazo}
\usepackage{enumitem}

\newcommand{\R}{\mathbb{R}}
\newcommand{\C}{\mathbb{C}}
\newcommand{\s}{\mathbb{S}}
\newcommand{\h}{\mathbb{H}}
\newcommand{\E}{\mathbb{E}}

\newcommand{\Nil}{\mathrm{Nil}_3}
\newcommand{\X}{\mathfrak{X}}
\newcommand{\df}{\mathrm{d}}

\newtheorem{theorem}{Theorem}[section]
\newtheorem{proposition}[theorem]{Proposition}
\newtheorem{corollary}[theorem]{Corollary}
\newtheorem{lemma}[theorem]{Lemma}

\theoremstyle{definition}
  \newtheorem{definition}[theorem]{Definition}
    \newtheorem{example}[theorem]{Example}

\theoremstyle{remark}
  \newtheorem{remark}[theorem]{Remark}

\numberwithin{equation}{section}

\setlist[itemize]{parsep=0.4em}
\setlist[enumerate]{parsep=0.4em}
\title{Compact stable surfaces with constant mean curvature in Killing submersions}

\author{Ana M. Lerma}
\address{Departamento de Didáctica de las Ciencias \\
Universidad de Ja\'{e}n \\
23071 Jaén, SPAIN} 
\email{alerma@ujaen.es}

\author{Jos\'{e} M. Manzano}
\address{Department of Mathematics \\
King's College London \\
WC2R 2LS London, UK} 
\email{manzanoprego@gmail.com}

\thanks{}



\begin{document}

\begin{abstract}
A Killing submersion is a Riemannian submersion from a $3$-manifold to a surface, both connected and orientable, whose fibres are the integral curves of a Killing vector field, not necessarily unitary. The first part of this paper deals with the classification of all Killing submersions in terms of two geometric functions, namely the bundle curvature and the length of the Killing vector field, which can be prescribed arbitrarily. In a second part, we show that if the base is compact and the submersion admits a global section, then it also admits a global minimal section. These turn out to be the only global sections with constant mean curvature, which solves the Bernstein problem in Killing submersions over compact base surfaces, as well as the Plateau problem with empty boundary. Finally, we prove that any compact orientable stable surface with constant mean curvature immersed in the total space of a Killing submersion must be either an entire minimal section or everywhere tangent to the Killing direction.
\end{abstract}

\maketitle

\section{Introduction}

Surface theory in Riemannian $3$-manifolds admitting a Killing vector field has experienced an increasing interest during the last decades. This interest has been specially significant in the case of product spaces $M\times\R$ as well as in homogeneous $3$-manifolds with isometry group of dimension $4$, also known as $\E(\kappa,\tau)$-spaces. As a common framework for these spaces, the theory of unit-Killing submersions was developed, being the work of Rosenberg, Souam and Toubiana~\cite{RST} pioneer in this topic (see also~\cite{EO,LR,MO,SV}). 

The first classification results for unit-Killing submersions were obtained by the second author. It is proved in~\cite{Man14} that, given a simply connected surface $M$ and a function $\tau\in\mathcal{C}^\infty(M)$, there exists a Riemannian submersion $\pi:\E\to M$ such that $\E$ is simply connected and orientable, the fibres of $\pi$ are the integral curves of a unit Killing vector field $\xi$ and $\pi$ has bundle curvature $\tau$ (i.e., $\overline\nabla_X\xi=\tau X\wedge\xi$ for all vector fields $X\in\X(\E)$, where $\overline\nabla$ stands for the Levi-Civita connection in $\E$). Moreover, $\pi$ turns out to be unique under these assumptions. Dropping the unitary condition, a Riemannian submersion fulfilling the aforementioned properties will be called a \emph{Killing submersion} (see Definition~\ref{def:killing}).

This generalization is motivated by the fact that some $3$-manifolds admitting non-unitary Killing vector fields have recently gained certain relevance. Among them we highlight the simply connected homogeneous ones (we refer the reader to the survey by Meeks and Pérez~\cite{MeeksPerez} and the references therein), whose spaces of Killing vector fields have dimension at least $3$, giving rise to Killing-submersion structures in any given Killing direction (see Examples~\ref{ex:homogeneous-R3} and~\ref{ex:homogeneous-S3}). Although we will give a complete classification of Killing submersions, it is also important to mention that several authors have already dealt with surfaces in generic $n$-manifolds admitting a Killing vector field (see, for instance,~\cite{ADR,DL}).

Given an arbitrary Killing submersion $\pi:\E\to M$, we can associate to $\pi$ two geometric functions. One of them is the \emph{bundle curvature} $\tau$, a function determining the horizontal part of the curvature $2$-form associated to $\pi$ (see Equation~\eqref{eqn:tau}), which generalizes the definition of bundle curvature in the unitary case and measures the non-integrability of the horizontal distribution. The second function is the \emph{Killing length} $\mu$, i.e., the length of the Killing vector field. Both $\tau$ and $\mu$ are constant along fibres and induce smooth functions on $M$. This way we can pose the problem of existence and uniqueness of Killing submersions fixing a Riemannian surface $M$ and the functions $\tau,\mu\in\mathcal{C}^\infty(M)$, $\mu>0$. 

Section~\ref{sec:Killing-submersions} is devoted to give the following complete solution:
\begin{itemize}
	\item If $M$ is simply connected, then there exists a Killing submersion $\pi:\E\to M$ with bundle curvature $\tau$ and Killing length $\mu$, which is unique provided that $\E$ is simply connected.
	\begin{itemize}
	 	\item If $M$ is topologically $\R^2$, then $\pi$ is a trivial fibration, and we are able to get an explicit model for $\pi$ (see Theorem~\ref{thm:classification-disk} and its proof). In particular, $\E$ is diffeomorphic to $\R^3$.
	 	\item If $M$ is topologically $\s^2$, then $\pi$ admits a global section if and only if $\int_M\frac{\tau}{\mu}=0$, and in that case $\E$ is diffeomorphic to $\s^2\times\R$. Otherwise $\pi$ is topologically the Hopf fibration and $\E$ is diffeomorphic to $\s^3$ (see Theorem~\ref{thm:classification-sphere}).
	 \end{itemize}
	 \item If $\pi:\E\to M$ is a Killing submersion and $M$ or $\E$ are not simply connected, then there exist a Killing submersion $\widetilde\pi:\widetilde\E\to\widetilde M$, being $\widetilde M$ and $\widetilde\E$ the universal Riemannian coverings of $M$ and $\E$, respectively, and a group $G$ of isometries on $\widetilde\E$ preserving the Killing direction, such that $G$ acts properly discontinuously on $\widetilde\E$ and $\E=\widetilde\E/G$ (see Theorem~\ref{thm:classification-other}).
\end{itemize}
Existence is therefore guaranteed by the above results, but uniqueness may fail when the base surface is not simply connected (see Example~\ref{ex:non-uniqueness}). 

It is also important to remark that Killing submersions are well-understood objects at the level of Differential Topology (see, for instance,~\cite{GHV,Ste}), so our contribution is to give a uniform treatment to all of them in terms of the functions $\tau$ and $\mu$, which prove to be useful when analysing the geometry and topology of surfaces. One example is the duality between mean curvature and bundle curvature shown by the Calabi-type correspondence~\cite[Theorem~3.7]{LeeMan}; another example is the fact that $\pi:\E\to M$ admits a global section if and only if either $M$ is non-compact or $M$ is compact and $\int_M\frac{\tau}{\mu}=0$ (see Proposition~\ref{prop:existence-sections}).

In section~\ref{sec:mean-curvature}, we will introduce graphs in Killing submersions as sections of the submersion over open sets in the base, and give a divergence-type equation for their mean curvature (see Lemma~\ref{lemma:H}). We will prove that any Killing submersion $\pi:\E\to M$ over a compact surface $M$ which admits a global section also admits a global \emph{minimal} section (see Theorem~\ref{thm:existence-entire-minimal}). The proof is based on minimizing area inside the isotopy class of a global section by means of a result of Meeks, Simon and Yau~\cite{MeeksSimonYau}. This argument is not valid if $M$ is a sphere, but in that case we can employ a Calabi-type duality together with an existence result of Gerhardt~\cite{Ger} for graphs with prescribed mean curvature in Lorentzian warped products. The aforesaid duality is inspired by the work of Albujer and Al\'{i}as~\cite{AA} (see also~\cite[Theorem~3.7]{LeeMan}). Let us now give two ways of understanding Theorem~\ref{thm:existence-entire-minimal}:
\begin{itemize}
	\item  On the one hand, Theorem~\ref{thm:existence-entire-minimal} solves the Bernstein problem in any Killing submersion $\pi:\E\to M$ provided that $M$ is compact: entire graphs with constant mean curvature $H\in\R$ exist if and only if $H=0$ and $\int_M\frac{\tau}{\mu}=0$, and they are unique up to isometries.
	\item  On the other hand, Theorem~\ref{thm:existence-entire-minimal} can be regarded as the solution to a Plateau problem with empty boundary, which from the point of view of Calculus of Variations is equivalent to minimizing the area functional
	\begin{equation}\label{eqn:area-functional}
	  \mathcal{A}(u)=\int_M\sqrt{\rho^2+\|\nabla u+X\|^2},
	\end{equation}
	where $M$ is a compact surface, $\rho\in\mathcal{C}^\infty(M)$ is positive, and $X$ is a smooth vector field in $M$. Theorem~\ref{thm:existence-entire-minimal} guarantees the existence of $u\in\mathcal{C}^\infty(M)$ minimizing~\eqref{eqn:area-functional} among all Lipschitz functions on $M$ (note that~\eqref{eqn:area-functional} is the area of an entire graph in a Killing submersion over $M$ with Killing length $\mu=\rho^{-1}$ and bundle curvature $\tau=2\rho^{-1}\,\mathrm{div}(JX)$, see Equation~\eqref{eqn:area-element}).
\end{itemize}
Using the work of Simon on the existence of embedded minimal spheres~\cite{Simon} we also deduce that a Killing submersion $\pi:\E\to M$, $M$ not necesarily compact, admits an immersed minimal sphere if and only if $M$ is a sphere itself (see Theorem~\ref{thm:minimal-spheres}). These minimal spheres are unique up to ambient isometries if the submersion is trivial by Theorem~\ref{thm:existence-entire-minimal}, and also if $\E$ is a homogeneous $3$-sphere~\cite{MeeksMiraPerezRos}, but uniqueness is not hitherto known to hold in general.

In Section~\ref{sec:stability} we will address the problem of stability of a compact orientable surface with constant mean curvature immersed in the total space of a Killing submersion $\pi:\E\to M$. Recall that a two-sided surface $\Sigma$ immersed in $\E$ (here two-sided is equivalent to orientable, since $\E$ is assumed orientable) has constant mean curvature $H\in\R$ if and only if it is a critical point of $\mathcal J=\mathrm{Area}-2H\cdot\mathrm{Vol}$, as proved by Barbosa, do Carmo and Eschenburg~\cite{BCE}. Then $\Sigma$ is said \emph{stable} if 
\[\mathcal J''(0)=-\int_\Sigma\left(\Delta f+(|A|^2+\mathrm{Ric}(N))f\right)\geq 0,\]
for all compactly-supported smooth functions $f\in\mathcal C_0^\infty(\Sigma)$. In other words, $\Sigma$ is a second-order minimum of $\mathcal J$ for all compactly-suppor\-ted normal variations of $\Sigma$ via the trivialization of the normal bundle given by a unit normal vector field $N$. Here $A$ denotes the shape operator. This is not the notion of stability associated to the isoperimetric problem, where the variations are required to be volume-preserving, but a stronger one (a comprehensive introduction to stability can be found in the survey by Meeks, Pérez and Ros~\cite{MeeksPerezRos}). We will use a characterization due to Fischer-Colbrie~\cite{FC} stating that $\Sigma$ is stable if and only if there exists a smooth function $u>0$ satisfying $Lu=0$, where $L=\Delta+|A|^2+\mathrm{Ric}(N)$ is the so-called \emph{stability operator} of $\Sigma$ (a more precise expression for this Schr\"{o}dinger operator in the Killing-submersion setting will be given in Lemma~\ref{lemma:L}).

Using the fact that the angle function $\nu=\langle N,\xi\rangle$ associated to a vertical Killing vector field $\xi$ lies in the kernel of $L$, we will show that any compact orientable stable surface immersed in $\E$ with constant mean curvature must be either a vertical cylinder (i.e., everywhere vertical) or an entire minimal graph (see Theorem~\ref{thm:compact-stable}). In particular, if $M$ is not compact and the fibres of $\pi$ are not compact either, then $\E$ does not admit compact orientable stable surfaces with constant mean curvature. Theorem~\ref{thm:compact-stable} improves previous results by Mero\~{n}o and Ortiz~\cite[Corollaries 4, 6 and 8]{MO} and generalizes other results by the second author, P\'{e}rez and Rodr\'{i}guez~\cite{MPR} in $\E(\kappa,\tau)$-spaces. Theorem~\ref{thm:compact-stable} also reveals the topology of a stable surface $\Sigma$ immersed in $\E$ with constant mean curvature $H\in\R$; for instance, if $\Sigma$ is not a torus, then $H=0$ and $\Sigma$ is an entire minimal graph.

Although our techniques only produce entire minimal graphs when the base surface is compact, we conjecture that any Killing submersion over a non-compact surface admits a global minimal section. This is related to the existence of global space-like sections with prescribed mean curvature in certain Lorentzian 3-manifolds (see~\cite{LeeMan}).

\noindent\textbf{Acknowledgement.} This research has been partially supported by the Spanish MEC-Feder research project MTM2014-52368-P. The second author was also supported by the EPSRC grant no.\ EP/M024512/1. The authors are grateful to Hojoo Lee and Joaqu\'{i}n P\'{e}rez for pointing out some insightful remarks leading to the final version of this manuscript.

\section{Killing submersions}\label{sec:Killing-submersions}

Let $\pi:\E\to M$ be a Riemannian submersion from a Riemannian $3$-manifold $\E$ to a Riemannian surface $M$, both of them connected and orientable. A vector $v\in T\E$ will be called \emph{vertical} when $v\in\ker(\mathrm{d}\pi)$ and \emph{horizontal} when $v\in\ker(\mathrm{d}\pi)^\bot$. Recall that $\pi$ is Riemannian if it preserves the length of horizontal vectors. 

\begin{definition}\label{def:killing}
$\pi:\E\to M$ is called a \emph{Killing submersion} if the fibres of $\pi$ are the integral curves of a complete Killing vector field $\xi\in\mathfrak{X}(\E)$ without zeroes.
\end{definition}

It is important to notice that $\xi$ is not unique under these conditions, since multiplying $\xi$ by a non-zero real constant also gives a Killing vector field without zeroes generating the same integral curves. 

Given a Killing submersion $\pi:\E\to M$ and fixing some $\xi$ satisfying Definition~\ref{def:killing}, we can consider the connection $1$-form $\alpha\in\Omega^1(\E)$, $\alpha(X)=\langle X,\xi\rangle$, and the curvature $2$-form $\omega=\tfrac{1}{2}\df\alpha$ given by $\omega(X,Y)=\langle\overline\nabla_X\xi,Y\rangle$ for all $X,Y\in\X(\E)$, being $\overline\nabla$ the Levi-Civita connection in $\E$. Since $\xi$ is Killing, we get that $\omega$ is skew-symmetric, and the function $\tau\in\mathcal{C}^\infty(\E)$ defined as
\begin{equation}\label{eqn:tau}
\tau(p)=\frac{-1}{\|\xi_p\|}\,\omega_p(e_1,e_2),
\end{equation}
where $\{e_1,e_2,\xi_p/\|\xi_p\|\}$ is a positively oriented orthonormal basis of $T_p\E$, does not depend on the choice of $\{e_1,e_2\}$ or $\xi$. Moreover, if $\xi$ is unitary, then $\tau$ satisfies the well-known identity $\overline\nabla_X\xi=\tau X\wedge\xi$ for all $X\in\X(\E)$, so $\tau$ will be called the \emph{bundle curvature} of the Killing submersion, extending previous definitions in the unitary case~\cite{EO,Man14,RST,SV}. The bundle curvature also measures the non-integrability of the horizontal distribution associated to $\pi$, in the sense that $\tau=0$ if and only if $\ker(\mathrm{d}\pi)^\bot$ is integrable (see Equation~\eqref{eqn:tau-model} below). 

The $1$-parameter group of isometries associated to $\xi$ will be denoted by $\{\phi_t\}$, and its elements will be called \emph{vertical translations}. Note that $\phi_t$ is defined for all values of $t\in\R$ since $\xi$ is assumed complete. Due to the fact that $\phi_t:\E\to\E$ is an isometry such that $(\phi_t)_*\xi=\xi$ and $(\phi_t)_*\omega=\omega$, the bundle curvature is constant along the fibres of $\pi$, and so is the \emph{Killing length} $\mu=\|\xi\|\in\mathcal C^\infty(\E)$. It follows that both $\tau$ and $\mu$ induce functions in $M$ that will be also denoted by $\tau,\mu\in\mathcal C^\infty(M)$.

\begin{definition}
Two Killing submersions $\pi:\E\to M$ and $\pi':\E'\to M$, over the same base surface $M$, are \emph{isomorphic} if there exists an isometry $T:\E\to\E'$ such that $\pi'\circ T=\pi$. 
\end{definition}

Given a base surface $M$ and two functions $\tau,\mu\in\mathcal{C}^\infty(M)$, $\mu>0$, our first goal is to classify (up to isomorphism) all Killing submersions over $M$ with bundle curvature $\tau$ and Killing length $\mu$. This goal was achieved in the unitary case, provided that $M$ is simply connected in~\cite{Man14}. Nonetheless, it is important to point out some differences in the non-unitary case:
\begin{itemize}
	\item The fibres of $\pi$ are not geodesic in general. As shown in Equation~\eqref{eqn:levi-civita} below, if $\gamma$ is a unit-speed parametrization of a fibre, then $\overline\nabla_{\gamma'}\gamma'=-\frac{1}{\mu}\overline\nabla\mu$, so geodesic fibres correspond to the critical points of $\mu$. 
	\item If a fibre of $\pi$ has finite length, then all fibres of $\pi$ have finite length, but this length may vary from fibre to fibre. In spite of that, there exists a minimal $\ell>0$ such that $\phi_\ell:\E\to\E$ is the identity map. This value of $\ell$ is the analogue to the constant length of the fibre in the non-unitary case.
\end{itemize}
Note also that we can change the metric in the vertical direction preserving the submersion structure and making the Killing vector field unitary, so results guaranteeing the existence of global sections in the unitary case still hold. For instance, if $M$ is not compact, then any Killing submersion $\pi:\E\to M$ admits a global section~\cite[Section VIII.5]{GHV}. The same conclusion holds if the fibres of $\pi$ have infinite length~\cite[Theorem 12.2]{Ste}. A concise characterization of the existence of global sections is given in Proposition~\ref{prop:existence-sections} below.

Next we exhibit some distinguished examples of Killing submersions.

\begin{example}\label{ex:warped-product}
If $\tau=0$, then the horizontal distribution is integrable and we obtain the warped product $M\times_\mu\R$ with $1$-dimensional fibres, which is the product manifold $M\times\R$ endowed with the Riemannian metric $\pi_M^*(\df s_M^2)+\mu^2\pi_\R^*(\df t^2)$, where $\mu$ is a function not depending on $t$, and $\pi_M$ and $\pi_\R$ are the usual projections. If $\mu$ is constant, we get the Riemannian product space $M\times\R$.
\end{example}

\begin{example}\label{ex:homogeneous-R3}
Every homogeneous $3$-manifold $X$ homeomorphic to $\R^3$ is isometric to the semi-direct product $\R^2\ltimes_A\R$ (for some real $2\times 2$ matrix $A$) equipped with some left-invariant metric (see~\cite{MeeksPerez} for a detailed description of these metrics). This metric can be chosen such that the left-invariant frame
\begin{align*}
 E_1&=\alpha_{11}(z)\partial_x+\alpha_{21}(z)\partial_y,&E_2&=\alpha_{12}(z)\partial_x+\alpha_{22}(z)\partial_y,&E_3&=\partial_z,
\end{align*}
is orthonormal (here $\alpha_{ij}(z)$ denote the entries of the exponential matrix $e^{zA}$ and $(x,y,z)$ represents the usual coordinates in $\R^3$). The vector field $\partial_x$ is right-invariant and Killing, and its integral curves are the fibres of the Killing submersion $(x,y,z)\mapsto(y,z)$. The metric can be expressed as
\[\frac{1}{\alpha_{22}^2+\alpha_{21}^2}\df y^2+\df z^2+\frac{\alpha_{22}^2+\alpha_{21}^2}{(\alpha_{11}\alpha_{22}-\alpha_{12}\alpha_{21})^2}\left(\df x-\frac{\alpha_{11}\alpha_{21}+\alpha_{12}\alpha_{22}}{\alpha_{22}^2+\alpha_{21}^2}\df y\right)^2.\]
The base surface is $\R^2$ equipped with the metric $(\alpha_{22}^2+\alpha_{21}^2)^{-1}\df y^2+\df z^2$, and the bundle curvature and Killing length are determined by
\begin{align*}
\frac{2\tau}{\mu}&=\frac{\partial}{\partial z}\!\left(\frac{\alpha_{11}\alpha_{21}+\alpha_{12}\alpha_{22}}{\alpha_{22}^2+\alpha_{21}^2}\right),&\mu&=\sqrt{\frac{\alpha_{22}^2+\alpha_{21}^2}{\alpha_{11}\alpha_{22}-\alpha_{12}\alpha_{21}}}.
\end{align*}
We emphasize here that different choices of the Killing vector field in $X$ give rise to non-isomorphic Killing submersion structures in $X$.
\end{example}

\begin{example}\label{ex:homogeneous-S3}
If $X$ is an homogeneous $3$-manifold homeomorphic to $\s^3$, then $X$ is isometric to the $3$-dimensional Lie group $\mathrm{SU}(2)$ eqquiped with some left-invariant metric, and any right-invariant $\xi\in\X(X)$ is Killing. If $\xi$ is not identically zero, then it has no zeros and its integral curves are compact~\cite{MeeksMiraPerezRos}. The space of fibres $M$ is topologically $\s^2$ and can be easily endowed with the structure of a Riemannian surface such that the natural projection from $X$ to $M$ is a Killing submersion (it suffices to induce in $M$ the metric in the distribution orthogonal to $\xi$). This submersion is topologically the Hopf fibration $\pi_{\mathrm{H}}:\s^3\to\s^2$.
\end{example}

Examples~\ref{ex:homogeneous-R3} and~\ref{ex:homogeneous-S3} cover all simply connected homogeneous $3$-manifolds, except for the Riemannian products $\s^2(\kappa)\times\R$, $\kappa>0$, which do not admit any Lie group structures (see~\cite[Theorem 2.4]{MeeksPerez}). Nevertheless the usual projection $\s^2(\kappa)\times\R\to\s^2(\kappa)$ is a Killing submersion with $\tau=0$ and constant $\mu$.

\subsection{Killing submersions over a disk} We obtain the following direct generalization of Theorems~2.8 and~4.2 in~\cite{Man14}. We will give a detailed proof since it provides a constructive and explicit way of producing Killing submersions with prescribed bundle curvature and Killing length.

\begin{theorem}[Local classification of Killing submersions]\label{thm:classification-disk}
Let $M$ be a non-compact simply connected Riemannian surface and let $\tau,\mu\in\mathcal C^\infty(M)$, $\mu>0$. Then, there exists a Killing submersion $\pi:\E\to M$ such that
\begin{enumerate}
 \item the fibres of $\pi$ have infinite length,
 \item $\tau$ is the bundle curvature of $\pi$, and
 \item $\mu$ is the length of a Killing field $\xi$ whose integral curves are the fibres of $\pi$.
\end{enumerate}
Moreover, such a Killing submersion $\pi$ is unique up to isomorphism.
\end{theorem}

\begin{proof}
The fact that $M$ is non-compact and simply connected means that there is an isometry $\varphi:(\Omega,\df s^2_\lambda)\to M$, where $\Omega\subset\R^2$ is the unit disk or the whole $\R^2$ and $\mathrm{d}s^2_\lambda=\lambda^2(\df x^2+\df y^2)$ for some positive $\lambda\in\mathcal C^\infty(\Omega)$. 

Let us suppose that $\pi:\E\to M$ is a Killing submersion satisfying (1), (2) and (3). The condition (1) yields the existence of a global smooth section $F_0:M\to\E$ (see~\cite[Theorem 12.2]{Ste}), and it also implies that
\[\Psi:\Omega\times\R\to\E,\quad \Psi((x,y),t)=\phi_t(F_0(\varphi(x,y)))\]
is a global diffeomorphism, so the following diagram is commutative
\[
\xymatrix{\Omega\times\R\ar^{\Psi}[rr]\ar_{\pi_1}[d]&&\E\ar^{\pi}[d]\\
\Omega\ar^{\varphi}[rr]&& M
}
\]
where $\pi_1:\Omega\times\R\to\Omega$ is the projection over the first factor. Now we can induce in $\Omega\times\R$ a metric $\mathrm{d}s^2$ making $\Psi$ an isometry, so $\pi_1$ becomes a Killing submersion over $(\Omega,\df s_\lambda^2)$. The orthonormal frame $\{e_1=\frac{1}{\lambda}\partial_x,e_2=\frac{1}{\lambda}\partial_y\}$ in $(\Omega,\df s_\lambda^2)$ lifts via $\pi_1$ to a global frame $\{E_1,E_2\}$ of the horizontal distribution associated to $\pi_1$. Hence $\{E_1,E_2,E_3=\frac{1}{\mu}\partial_t\}$ forms a global orthonormal frame of $\E$. Since $\pi_1(x,y,t)=(x,y)$, there exist $a,b\in\mathcal{C}^\infty(\Omega)$ such that we can write
\begin{equation}\label{eqn:frame}
\begin{aligned}
(E_1)_{(x,y,t)}&=\tfrac{1}{\lambda(x,y)}\,\partial_x+a(x,y)\partial_t,\\
(E_2)_{(x,y,t)}&=\tfrac{1}{\lambda(x,y)}\,\partial_y+b(x,y)\partial_t,\\
(E_3)_{(x,y,t)}&=\tfrac{1}{\mu(x,y)}\,\partial_t.
\end{aligned}
\end{equation}
In other words, the metric $\mathrm{d}s^2$ can be expressed in coordinates as
\begin{equation}\label{eqn:metric}
\df s^2=\lambda^2(\df x^2+\df y^2)+\mu^2\left(\mathrm{d} t-\lambda(a\df x+b\df y)\right)^2.
\end{equation}
Note that $\partial_t=\mu E_3$ represents the vertical Killing vector field, and there is no loss of generality in supposing that $\{E_1,E_2,E_3\}$ is positively oriented in $\E$.

By means of Equation~\eqref{eqn:tau}, the bundle curvature $\tau$ of $\pi_1$ reads
\begin{equation}\label{eqn:tau-model-aux}
\tau=\tfrac{-1}{\mu}\omega(E_1,E_2)=-\langle\nabla_{E_1}E_3,E_2\rangle=\langle\nabla_{E_1}E_2,E_3\rangle.
\end{equation}
Likewise, using $E_2$ rather than $E_1$, we get $\tau=-\langle\nabla_{E_2}E_1,E_3\rangle$, and adding this to~\eqref{eqn:tau-model-aux}, we finally obtain
\begin{equation}\label{eqn:tau-model}
\tau=\tfrac{1}{2}\langle[E_1,E_2],E_3\rangle=\tfrac{\mu}{2\lambda^2}\!\left((\lambda b)_x-(\lambda a)_y\right).
\end{equation}
Note that the Lie bracket that can be easily computed from~\eqref{eqn:frame} as
\begin{equation}\label{eqn:bracket}
\begin{aligned}~
[E_1,E_2]&=\tfrac{\lambda_y}{\lambda^2}E_1-\tfrac{\lambda_x}{\lambda^2}E_2+\tfrac{\mu}{\lambda^2}\left((\lambda b)_x-(\lambda a)_y\right)E_3,\\
[E_1,E_3]&=\tfrac{-\mu_x}{\lambda\mu}E_3,\qquad\qquad [E_2,E_3]=\tfrac{-\mu_y}{\lambda\mu}E_3.
\end{aligned}
\end{equation}
Using this representation of Killing submersions, the proof follows easily.

\textbf{Uniqueness.} Let us suppose that $\pi':\E'\to M$ is another Killing submersion satisfying the same conditions and construct likewise an isometry $\Psi':\Omega\times\R\to\E'$ which induces functions $a',b'\in\mathcal C^\infty(\Omega)$ as above. Since $\tau$ and $\mu$ coincide for both submersions, Equation~\eqref{eqn:tau-model} yields $(\lambda b')_x-(\lambda a')_y=(\lambda b)_x-(\lambda a)_y$. Equivalently, the following identity holds in $\Omega$:
\[(\lambda(b'-b))_x=(\lambda(a'-a))_y.\] 
As $\Omega$ is simply connected, Poincar\'e's lemma guarantees the existence of $d\in\mathcal C^\infty(\Omega)$ such that $\lambda(b'-b)=d_y$ and $\lambda(a'-a)=d_x$. It is not difficult to check that $R:\Omega\times\R\to\Omega\times\R$ given by $R((x,y),t)=((x,y),t-d(x,y))$ satisfies that $T=\Psi'\circ R\circ\Psi^{-1}:\E\to\E'$ is an isometry such that $\pi'\circ T=\pi$.

\textbf{Existence.} We will show an explicit metric realizing the conditions in the statement, whose expression is inspired by the ideas above. Let us take 
\begin{equation}\label{eqn:models}
\E=\bigl(\Omega\times\R,\df s^2=\lambda^2(\df x^2+\df y^2)+\mu^2(\df t+\eta(y\df x-x\df y))^2\bigr),
\end{equation}
where
\begin{equation}\label{eqn:eta}
\eta=\eta(x,y)=\int_0^1\frac{2s\,\tau(sx,sy)\,\lambda(sx,sy)^2}{\mu(sx,sy)}\df s.
\end{equation}
Then $\pi:\E\to(\Omega,\df s_\lambda^2)$ given by $\pi(x,y,z)=(x,y)$ becomes a Riemannian submersion with Killing vector field $\partial_t$ (the coefficients of $\df s^2$ do not depend upon $t$) such that $\|\partial_t\|=\mu$. Finally~\eqref{eqn:models} arises from~\eqref{eqn:metric} for $a=\frac{-y\eta}{\lambda}$ and $b=\frac{x\eta}{\lambda}$. Hence~\eqref{eqn:tau-model} tells us that the bundle curvature associated to~\eqref{eqn:models} is $\frac{\mu}{2\lambda^2}((x\eta)_x+(y\eta)_y)$, which is equal to $\tau$ by a simple calculation using~\eqref{eqn:eta}. 
\end{proof}

\begin{remark}\label{rmk:classification-disk-general}
Condition (1) in the statement is not restrictive at all. If the fibres of $\pi:\E\to M$ have finite length, then $\pi$ can be recovered as a Riemannian quotient of a Killing submersion with fibres of infinite length under an appropriate vertical translation. This is a consequence of Theorem~\ref{thm:classification-other} below. In particular, the total space of a Killing submersion over a disk is diffeomorphic to $\R^3$ or to $\R^2\times\s^1$.

This also implies that the condition (1) in the statement may be replaced by assuming that $\E$ is simply connected.
\end{remark}

\subsection{Killing submersions over the sphere}
The classification result when the base surface is topologically $\s^2$ can be carried out in a similar fashion to the unitary case (see~\cite[Propositions 4.5 and 4.9]{Man14}). The key ingredients are the horizontal lifts of curves and their holonomy properties, which we formulate next.

Let $\pi:\E\to M$ be a Killing submersion and $\alpha:[a,b]\to M$ a curve of class $\mathcal{C}^1$. A horizontal lift of $\alpha$ is a curve $\widetilde\alpha:[a,b]\to\E$ of class $\mathcal{C}^1$ such that $\widetilde\alpha'$ is always horizontal and $\pi\circ\widetilde\alpha=\alpha$ in $[a,b]$. The horizontal lift always exists and it is unique when the point $\widetilde\alpha(a)$ is chosen in the fibre of $\alpha(a)$.

\begin{proposition}\label{prop:holonomia}
Let $\pi:\E\to M$ be a Killing submersion whose fibres have infinite length and let $\alpha:[a,b]\to M$ be a simple $\mathcal{C}^1$-curve bounding an orientable relatively compact open set $O\subset M$. Given a horizontal lift $\widetilde\alpha$ of $\alpha$, there exists a unique $d\in\R$ such that $\phi_d(\widetilde\alpha(a))=\widetilde\alpha(b)$ and it satisfies
\[\left|\int_O\frac{2\tau}{\mu}\right|=|d|.\]
\end{proposition}

It is worth highlighting that the proof of Proposition~\ref{prop:holonomia} relies on the divergence theorem applied to the divergence-type expression for $\tau$ given by Equation~\eqref{eqn:tau} (see also~\cite[Proposition 3.3]{Man14}).

\begin{theorem}[Classification of Killing submersions over the sphere]\label{thm:classification-sphere}
Let $M$ be $\s^2$ endowed with some Riemannian metric, and consider $\tau,\mu\in\mathcal C^\infty(\s^2)$, with $\mu>0$. Then there exists a Killing submersion $\pi:\E\to M$ such that
\begin{enumerate}
 \item $\E$ is simply connected,
 \item $\tau$ is the bundle curvature of $\pi$, and
 \item $\mu$ is the length of a Killing field $\xi$ whose integral curves are the fibres of $\pi$.
\end{enumerate}
Moreover, such a Killing submersion $\pi$ is unique up to isomorphism.
\begin{itemize}
 \item [(a)] If $\int_M\frac{\tau}{\mu}=0$, then the length of the fibres of $\pi$ is infinite and $\pi$ is isomorphic to
 \[\pi_1:(\s^2\times\R,\df s^2)\to\s^2,\qquad \pi_1(p,t)=p,\]
 for some Riemannian metric $\mathrm{d}s^2$, with Killing vector field $\xi_{(p,t)}=\partial_t$.
 \item[(b)] If, on the contrary, $\int_M\frac{\tau}{\mu}\neq 0$, then the fibres of $\pi$ have finite length and $\pi$ is isomorphic to the Hopf fibration 
 \[\pi_{\mathrm{H}}:(\s^3,\df s^2)\to\s^2,\qquad \pi_{\mathrm{H}}(z,w)=(2z\bar w,|z|^2-|w|^2),\]
 for some Riemannian metric $\mathrm{d}s^2$ in $\s^3$ with Killing vector field $\xi_{(z,w)}=(iz,iw)$. Here $\s^3$ and $\s^2$ are the unit spheres in $\C^2$ and $\R^3\equiv\C\times\R$, respectively.
\end{itemize}
\end{theorem}

\begin{remark}\label{rmk:classification-sphere-general}
If condition (1) in the statement of Theorem~\ref{thm:classification-sphere} is dropped, then we obtain a quotient under a vertical translation. This follows from Theorem~\ref{thm:classification-other} below. In the trivial case, the total space of a Killing submersion over a sphere is diffeomorphic to $\s^2\times\R$ or $\s^2\times\s^1$. In the non-trivial case the quotient must be taken under a translation of appropriate length, from where the total space must be diffeomorphic to $\s^3$ or to the lens space $L(n,1)$, $n\geq 2$. Note that these quotients of $\s^3$ are orientable ($L(2,1)$ is nothing but the real projective space $\mathbb{RP}^3$).
\end{remark}

\subsection{Killing submersions over a non-simply connected surface}

Let $\pi:\E\to M$ be a Killing submersion over an arbitrary orientable surface $M$ with bundle curvature $\tau$ and Killing length $\mu$. Our goal is to show that $\pi$ can be regarded as the quotient of a Killing submersion over a simply connected surface, which have been classified in the previous sections.

Let $\rho:\widetilde M\to M$ be the Riemannian universal covering of $M$, and consider $\widetilde\tau=\tau\circ\rho$ and $\widetilde\mu=\mu\circ\rho$ the lift to $\widetilde M$ of $\tau$ and $\mu$, respectively. In view of Theorems~\ref{thm:classification-disk} and~\ref{thm:classification-sphere}, there exists a unique Killing submersion $\widetilde\pi:\widetilde\E\to\widetilde M$ with bundle curvature $\widetilde\tau$, Killing length $\widetilde\mu$, and such that $\widetilde\E$ is simply connected. Let us denote by $\mathrm{Aut}(\rho)$ the group of automorphisms of $\rho$ (i.e., $\mathrm{Aut}(\rho)$ consists of those isometries $h:\widetilde M\to \widetilde M$ such that $\rho\circ h=\rho$). It is well known that $M$ is the Riemannian quotient of $\widetilde M$ under the properly discontinuous action of $\mathrm{Aut}(\rho)$. 

We will use $\mathrm{Aut}(\rho)$ to construct a group $G$ acting properly discontinuously by isometries on $\widetilde\E$, and such that $\E=\widetilde\E/G$, so we will need a way of producing isometries in the total space $\widetilde\E$ by lifting isometries in $\mathrm{Aut}(\rho)$. We will omit the proof of the following lemma since it is a direct generalization of~\cite[Theorem~2.8]{Man14}.

\begin{lemma}\label{lemma:levantamiento-isometrias}
Given $h\in\mathrm{Aut}(\rho)$ and $p,q\in\widetilde\E$ such that $h(\widetilde\pi(p))=\widetilde\pi(q)$, there exists a unique isometry $f:\widetilde\E\to\widetilde\E$ such that $\widetilde\pi\circ f=h\circ\widetilde\pi$ and $f(p)=q$.
\end{lemma}

\begin{remark}
Given a Killing submersion $\pi:\E\to M$ with Killing vector field $\xi$, an isometry $f$ preserving the direction $\xi$ is called a \emph{Killing isometry}. It is straightforward that any Killing isometry $f$ induces an isometry $h:M\to M$ such that $\tau\circ h=\pm\tau$ and such that $\mu\circ h=a\mu$ for some constant $a\neq 0$. Lemma~\ref{lemma:levantamiento-isometrias} ensures that given the isometry $h$ preserving $\tau$ and $\mu$, $f$ can be recovered up to composing with a vertical translation (indeed, it can be shown that the isometry $f$ in Lemma~\ref{lemma:levantamiento-isometrias} preserves the orientation in $\widetilde\E$).
\end{remark}

\begin{theorem}\label{thm:classification-other}
Let $\pi:\E\to M$ be a Killing submersion, and let $\rho:\widetilde M\to M$ and $\sigma:\widetilde\E\to\E$ be the universal Riemannian covering maps of $M$ and $\E$, respectively. 
\begin{enumerate}

\item[(a)] There exist a Killing submersion $\widetilde\pi:\widetilde\E\to\widetilde M$ such that $\rho\circ\widetilde\pi=\pi\circ\sigma$.
\item[(b)] There exists a group $G$ of Killing isometries acting properly discontinuously on $\widetilde\E$ such that $\E=\widetilde\E/G$.
\end{enumerate}
Moreover, each $f\in G$ is associated to some $h\in\mathrm{Aut}(\rho)$, so the diagram
\[\xymatrix{
	\widetilde\E\ar^{f}[rr]\ar_{\widetilde\pi}[d]&&\widetilde\E\ar^{\sigma}[rr]\ar_{\widetilde\pi}[d]&&\E\ar^{\pi}[d]\\
	\widetilde M\ar^{h}[rr]&&\widetilde M\ar^{\rho}[rr]&&M
}\]
is commutative. If the fibres of $\pi$ have infinite length, this correspondence between $G$ and $\mathrm{Aut}(\rho)$ is bijective. Otherwise, any two isometries associated to $h$ differ in a vertical translation of length a multiple of the length of the fibre.
\end{theorem}

\begin{proof}
Let us consider $\widetilde\pi:\widetilde\E\to\widetilde M$ as defined above. It suffices to construct the group $G$ acting properly discontinuously on $\widetilde\E$ and such that $\widetilde\E/G=\E$, so the projection to the quotient $\sigma:\widetilde\E\to\E$ will enjoy the desired properties.

Let us fix $p_0\in\widetilde\E$ and $q_0\in\E$ with $\pi(q_0)=\rho(\widetilde\pi(p_0))$. Given $h\in\mathrm{Aut}(\rho)$, we will begin by defining the Killing isometries associated to $h$. In order to find the initial conditions required by Lemma~\ref{lemma:levantamiento-isometrias}, horizontal lifts come in handy.

Let $\gamma:[0,1]\to\widetilde M$ be a smooth path with $\gamma(0)=\widetilde\pi(p_0)$ and $\gamma(1)=h(\widetilde\pi(p_0))$. Then $\alpha=\rho\circ\gamma$ is a closed curve in $M$ with $\alpha(0)=\alpha(1)=\rho(\widetilde\pi(p_0))$. Let us consider $\widehat\alpha:[0,1]\to\E$ the horizontal lift of $\alpha$ such that $\widehat\alpha(0)=q_0$, and take 
\[A(h)=\{t\in\R:\phi_t(q_0)=\widehat\alpha(1)\},\] where $\{\phi_t\}$ are the vertical translations in $\E$. For each $t\in A(h)$, Lemma~\ref{lemma:levantamiento-isometrias} yields the existence of an isometry $f:\widetilde\E\to\widetilde\E$ such that $\widetilde\pi\circ f=h\circ\widetilde\pi$ and $f(p_0)=\widetilde\phi_{-t}(\widehat\gamma(1))$, where $\{\widetilde\phi_t\}$ is the group of vertical translations in $\widetilde\E$ and $\widehat\gamma:[0,1]\to\widetilde\E$ is the horizontal lift of $\gamma$ such that $\widehat\gamma(0)=p_0$. This choice of the initial condition is motivated by the fact that the vertical distance between the endpoints of the horizontal lift of a closed curve must be preserved by these (local) isomorphisms of Killing submersions.

We define $G$ as the set of all the Killing isometries $f$ constructed above when $t\in A(h)$ and $h\in\mathrm{Aut}(\rho)$. It is easy to check that $G$ does not depend on the choice of the paths $\gamma$ as a consequence of Proposition~\ref{prop:holonomia}. Moreover, $G$ is a group acting properly discontinuously on $\widetilde\E$, with quotient space $\widetilde\E/G=\E$. Observe that, if there are two isometries $f_1$ and $f_2$ associated to the same $h$, then $f_1\circ f_2^{-1}=\widetilde\phi_t$ for some $t\in\R$ such that $\phi_t(q_0)=q_0$. Hence $t=0$ if the fibres have infinite length (so $f_1=f_2$) or $\widetilde\phi_t$ is a vertical translation of length a multiple of the length of the fibre, if this length is finite.
\end{proof}

Theorem~\ref{thm:classification-other} yields a constructive way of producing non-simply connected Killing submersions as Riemannian quotients of the simply connected ones by subgroups of Killing isometries. The simplest examples with non-trivial topology arise as quotients of the Heisenberg group, which we analyse next.

\begin{example}[Non-uniqueness]\label{ex:non-uniqueness}
Let us consider the Heisenberg group $\Nil(\tau)$, which admits a structure of Killing submersion over the Euclidean plane $\R^2$ with constant $\tau>0$ and $\mu=1$, so it is isometric to $\R^3$ with the metric given by~\eqref{eqn:models} for $\lambda=\mu$ and $\eta=\tau$. 

The isometries of $\Nil(\tau)$ defined by
\begin{align*}
f_1(x,y,z)&=(x+1,y,z+\tau y+a),&f_2(x,y,z)&=(x,y+1,z-\tau x+b),
\end{align*}
correspond to the translations $h_1(x,y)=(x+1,y)$ and $h_2(x,y)=(x,y+1)$, respectively. Observe that the commutator $f_1\circ f_2\circ f_1^{-1}\circ f_2^{-1}$ maps $(x,y,z)$ into $(x,y,z+2\tau)$, thus the group $G$ spanned by $f_1$ and $f_2$ has compact quotient $\E=\Nil(\tau)/G$ admitting a Killing submersion structure $\pi:\E\to\mathbb{T}=\R^2/H$, where $H$ is the group of isometries of $\R^2$ spanned by $h_1$ and $h_2$. Theorem~\ref{thm:classification-other} guarantees that these are all Killing submersions over the flat torus $\mathbb T$ with bundle curvature $\tau$ and unit Killing vector field.

Let us consider the closed curve $\alpha:[0,1]\to\mathbb T$, $\alpha(t)=(t,0)$, whose horizontal lift is $\widetilde\alpha(t)=(t,0,0)$. Since $(0,0,0)$ and $(1,0,a)$ are identified by $f_1$, the vertical distance from $\widetilde\alpha(0)$ to $\widetilde\alpha(1)$ is $2\tau-a$. As isomorphisms of Killing submersions preserve the horizontal lift of a curve as well as vertical distances, we conclude that different values of $a$ and $b$ give rise to non-isomorphic Killing submersions over $\mathbb T$ with the same bundle curvature and Killing length.
\end{example}

\section{The mean curvature equation}\label{sec:mean-curvature}

Let $\pi:\mathbb{E}\to M$ be a Killing submersion and fix a Killing vector field $\xi$ whose integral curves are the fibres of $\pi$, which will be supposed to have infinite length throughout this section.

Let $U\subseteq M$ be an open subset such that there exists a smooth section $F_0:U\to\mathbb{E}$. Given $u:U\to\R$, we define the \emph{Killing graph} associated to $u$ with respect to $F_0$ as the surface in $\mathbb E$ parametrized by
\[F_u:U\to\mathbb{E},\quad F_u(p)=\phi_{u(p)}(F_0(p)),\]
where $\{\phi_t\}$ is the $1$-parameter group of isometries associated to $\xi$. Next lemma gives a divergence-type formula for the mean curvature (generalizing the corresponding formula in the unitary case, see~\cite{LR,LeeMan}). We will assume that all functions are smooth in the sequel, despite the fact that most of the following arguments work for $\mathcal C^2$-graphs. 

\begin{lemma}\label{lemma:H}
Given $u\in\mathcal{C}^\infty(U)$, the mean curvature $H$ of $F_u$ with respect to a unit normal vector field $N$ along $F_u$, satisfies (as a function in $M$)
\[2H\mu =\mathrm{div}(\mu\,\pi_*N),\]
where $\mathrm{div}$ denotes the divergence on $M$ and $\mu=\|\xi\|$ is the Killing length.
\end{lemma}

\begin{proof}
It is well known that $2H=\overline{\mathrm{div}}(\overline{N})$, where $\overline{\mathrm{div}}$ is the ambient divergence and $\overline N$ is an extension to $\E$ of $N$. Here we will consider $\overline N$ as the unique extension of $N$ satisfying $(\phi_t)_*\overline{N}=\overline{N}$ for all $t\in\mathbb R$, i.e., $\overline N$ is constant along the fibres of $\pi$. Working locally if necessary, we can consider an orthonormal frame $\{e_1,e_2\}$ in $U$ and its horizontal lift $\{E_1,E_2\}$, which is completed to an orthonormal frame in $\pi^{-1}(U)$ by adding $E_3=\frac{1}{\mu}\xi$, and
\begin{equation}\label{mean-curvature1}
2H=\overline{\mathrm{div}}(\overline N)=\sum_{i=1}^2\langle\overline\nabla_{E_i}\overline N,E_i\rangle+\langle\overline\nabla_{E_3}\overline N,E_3\rangle.
\end{equation}
Since $\pi$ is Riemannian, $E_1$ and $E_2$ are horizontal and $\overline N$ is invariant under vertical translations, we obtain that $\sum_{i=1}^2\langle\nabla_{E_i}\overline N,E_i\rangle=\mathrm{div}_M(\pi_*N)$. Furthermore
\begin{equation}\label{mean-curvature2}
\langle\overline\nabla_{E_3}\overline{N},E_3\rangle=\langle[E_3,\overline{N}],E_3\rangle=\tfrac{1}{\mu^2}\left\langle[\xi,\overline{N}],\xi\right\rangle-\tfrac{1}{\mu}\bigl\langle\overline{N}\bigl(\tfrac{1}{\mu}\bigr)\xi,\xi\bigr\rangle=\tfrac{1}{\mu}\overline{N}(\mu).
\end{equation}
The first equality follows from Koszul formula and the fact that $\langle\overline N,E_3\rangle$ is constant along fibres, the second one follows from the fact that $E_3=\frac{1}{\mu}\xi$. In the third one, we used that $\xi$ is Killing to obtain that
\begin{align*}
\langle[\overline{N},\xi],\xi\rangle&=\langle\overline\nabla_{\overline{N}}\xi,\xi\rangle-\langle\overline\nabla_\xi \overline{N},\xi\rangle=-\langle\overline\nabla_\xi\xi,\overline{N}\rangle-\langle\overline\nabla_\xi \overline{N},\xi\rangle=-\xi(\langle \overline{N},\xi\rangle)=0.
\end{align*}
Plugging~\eqref{mean-curvature2} into~\eqref{mean-curvature1} and seeing $H$ and $\mu$ as functions on $M$ (recall that $\mu$ is constant along fibres), we get
\[2H=\mathrm{div}(\pi_*N)+\tfrac{1}{\mu}\overline N(\mu)=\mathrm{div}(\pi_*N)+\tfrac{1}{\mu}\langle\pi_*N,\nabla\mu\rangle_M=\tfrac{1}{\mu}\mathrm{div}\left(\mu\ \pi_*N\right),\]
so the statement follows.
\end{proof}

Now, in order to work out the term $\pi_*N$, let us consider $\overline u\in\mathcal C^\infty(\E)$ the extension of $u$  by making it constant along fibres, and $d\in\mathcal C^\infty(\E)$ the function that measures the signed vertical Killing distance from a point to the point in $F_0$ lying on the same fibre. In other words, $d$ is determined by the identity $\phi_{d(p)}(F_0(\pi(p)))=p$, for all $p\in\E$. Then the surface parametrized by $F_u$ is a level surface of the function $\overline u-d\in\mathcal C^\infty(\E)$ so we can compute $N=\overline\nabla(\overline u-d)/\|\overline\nabla(\overline u-d)\|$. If we set $Z=\pi_*(\overline\nabla d)$, it is not difficult to get to
\begin{equation}\label{eqn:N-proj}
\pi_*N=\frac{Gu}{\sqrt{\mu^{-2}+\|Gu\|^2}},\qquad Gu=\nabla u-Z,
\end{equation}
where the gradient and the norm are computed in $M$. The denominator in~\eqref{eqn:N-proj} is the area element of $F_u$, i.e., the area of the surface spanned by $u$ is equal to
\begin{equation}\label{eqn:area-element}
 \mathcal{A}(u)=\int_U \sqrt{\mu^{-2}+\|Gu\|^2}.
\end{equation}
The information about the bundle curvature as well as about the choice of the zero section $F_0$ is encoded in $Z$. 

\begin{lemma}\label{lemma:JZ}
The vector field $Z$ satisfies 
\[\mathrm{div}(JZ)=-\frac{2\tau}{\mu},\]
where $J$ denotes a $\frac{\pi}{2}$-rotation in $TM$.
\end{lemma}

\begin{proof}
As the computation is local, we can suppose the Killing submersion is the projection over the first factor $\pi_1:\Omega\times\R\to\Omega$, such that $\{e_1=\frac{1}{\lambda}\partial_x, e_2=\frac{1}{\lambda}\partial_y\}$ and $\{E_1,E_2,E_3\}$ given by~\eqref{eqn:frame} are orthonormal frames in $\Omega$ and $\Omega\times\R$, respectively, for some $a,b,\lambda\in\mathcal C^\infty(\Omega)$, $\lambda>0$. Taking the global initial section $F_0:\Omega\to\Omega\times\R$ as $F_0(x,y)=(x,y,0)$, we get that $\overline u(x,y,z)=u(x,y)$ and $d(x,y,z)=z$ in this model. Hence,
\[\overline\nabla d=\sum_{i=1}^3E_i(z)E_i=a E_1+b E_2+\tfrac{1}{\mu}E_3,\]
which gives $Z=\pi_*(\overline\nabla d)=a e_1+b e_2$. As $\{e_1,e_2\}$ is orthonormal (and can be assumed positively oriented), we deduce that $JZ=-be_1+ae_2=-\frac{b}{\lambda}\partial_x+\frac{a}{\lambda}\partial_y$. Using the expression for the divergence in a conformal metric and Equation~\eqref{eqn:tau-model} we finally obtain
\[\mathrm{div}(JZ)=\frac{1}{\lambda^2}\left(-(\lambda b)_x+(\lambda a)_y\right)=-\frac{2\tau}{\mu}.\qedhere\]
\end{proof}

Now we are able to characterize the existence of global sections.

\begin{proposition}[Existence of global sections]\label{prop:existence-sections}
Let $\pi:\E\to M$ be a Killing submersion with bundle curvature $\tau$ and Killing length $\mu$. The submersion $\pi$ admits a global section if and only if either $M$ is non-compact or $M$ is compact and $\int_{M}\frac{\tau}{\mu}=0$.
\end{proposition}

\begin{proof}
If $M$ is not compact, the existence follows from~\cite[Section VIII.5]{GHV}, so let us suppose that $M$ is compact. If $M$ admits a global section, then $\int_{M}\frac{\tau}{\mu}=0$ as a consequence of Lemma~\ref{lemma:JZ} and the divergence theorem. Conversely, let us assume that $\int_{M}\frac{\tau}{\mu}=0$ and take a regular Jordan curve $\Gamma\subset M$ separating $M$ in two components $M_1$ and $M_2$ such that $\int_{M_1}\frac{\tau}{\mu}=\int_{M_2}\frac{\tau}{\mu}=0$. By Proposition~\ref{prop:holonomia}, $\Gamma$ lifts to a horizontal closed curve $\widetilde\Gamma\subset\E$. Since each component $M_i$, $i\in\{1,2\}$, is not compact, there exist sections $F_i:\overline{M_i}\to\E$ such that $F_i(x)\in\widetilde\Gamma$ for all $x\in\Gamma$. Then $F: M\to\E$ given by $F(x)=F_i(x)$ if $x\in M_i$ is a global continuous section of $\pi$ which can be assumed smooth after a small perturbation around $\Gamma$.
\end{proof}

Before stating the existence of minimal sections, we need the following lemma, which collects some useful information given by the maximum principle for constant mean curvature surfaces, and will be used repeatedly in the sequel.

\begin{lemma}\label{lemma:H=0}
Let $\pi:\E\to M$ be a Killing submersion with bundle curvature $\tau$ and Killing length $\mu$. If $\Sigma$ is a compact surface immersed in $\E$ with constant mean curvature $H$ and everywhere transversal to the vertical direction, then $M$ is compact, $\int_M\frac{\tau}{\mu}=0$ and $\Sigma$ is an entire minimal graph. 

Moreover, if such an entire minimal graph exists, then it is unique up to vertical translations and minimizes area among all entire Lipschitz graphs in $\E$.
\end{lemma}

\begin{proof}
The absolute value of the Jacobian of $\pi_{|\Sigma}:\Sigma\to M$ equals $\frac{1}{\mu}|\langle N,\xi\rangle|$, where $N$ is a unit normal to the surface and $\xi$ is a vertical Killing vector field. Since $\Sigma$ is compact and transversal to $\xi$, we can assume that  $|\nu|\geq\varepsilon$ for some constant $\varepsilon>0$. As a consequence, $\pi_{|\Sigma}$ is a covering map with a finite number of sheets. In particular $M$ is compact and the bundle is trivial (i.e., it admits a global section), so we deduce that $\int_M\frac{\tau}{\mu}=0$ from Proposition~\ref{prop:existence-sections}. Therefore we will also assume that the fibres of $\pi$ have infinite length without losing generality. 

To prove that $\Sigma$ is an entire graph, let us assume by contradiction that $\pi_{|\Sigma}$ has more than one sheet. Then take $\Sigma'$ a vertical translation of $\Sigma$ with $\Sigma\cap\Sigma'=\emptyset$ and move $\Sigma'$ towards $\Sigma$ till there is a first contact point $p$. The normal vector fields to the surfaces coincide at $p$ because $\nu$ has a global sign. Since there is more than one sheet, the point $p$ comes from two different sheets of $\Sigma$, and the maximum principle implies that $\Sigma$ is invariant by a non-trivial vertical translation, contradicting the fact that it is compact and the fibres have infinite length. 

To prove that $\Sigma$ is minimal, Lemma~\ref{lemma:H} applied to the entire graph $\Sigma$ and the divergence theorem give $\int_M H\mu=0$, where $\mu$ is positive, so $H=0$. Uniqueness up to vertical translations follows from the maximum principle.

Finally, let us show the minimization property. Consider any entire Lipschitz graph $\Sigma'$, and suppose it does not intersect $\Sigma$ by applying a vertical translation. If $N$ and $N'$ stand for unit normal vector fields to $\Sigma$ and $\Sigma'$, respectively ($N'$ is defined almost everywhere), then we can extend $N$ to a unique $\overline N\in\X(\E)$ such that $(\phi_t)_*\overline N=\overline N$ for all $t$, where $\{\phi_t\}$ is the 1-parameter group of vertical translations. Then the ambient divergence $\overline{\mathrm{div}}(\overline N)$ identically vanishes since $\overline N$ is normal to a foliation of $\E$ by minimal surfaces, namely the translations of $\Sigma$. The divergence theorem for $\overline N$ in the Lipschitz region bounded by $\Sigma$ and $\Sigma'$ and Cauchy-Schwarz inequality yield
\[\mathrm{Area}(\Sigma)=\left|\int_\Sigma\langle N,N\rangle\right|=\left|\int_{\Sigma'}\langle N',N\rangle\right|\leq\mathrm{Area}(\Sigma'),\]
with equality if and only if $N'$ is collinear with $N$ almost everywhere, that is to say, if and only if $\Sigma'$ is a vertical translation of $\Sigma$.
\end{proof}

\begin{remark}
Lemma~\ref{lemma:H=0} can be extended to the case the mean curvature of $\Sigma$ does not change sign. Though the maximum principle cannot be applied directly because the mean curvatures of $\Sigma$ and $\Sigma'$ may not be properly ordered at the contact point, we could also consider and a Killing submersion over a Riemannian covering of $M$ where $\Sigma$ lifts to a entire graph. Since the mean curvature of the lifted surface does not change sign either, Lemma~\ref{lemma:H} gives the result.
\end{remark}

\begin{theorem}[Bernstein problem]\label{thm:existence-entire-minimal}
Let $\pi:\E\to M$ be a Killing submersion, and assume that $M$ is compact and $\int_M\frac{\tau}{\mu}=0$, where $\tau$ and $\mu$ denote the bundle curvature and the Killing length, respectively. Then $\E$ admits an entire minimal graph $\Sigma$.
\begin{enumerate}
	\item Up to a vertical translation $\Sigma$ is the the unique entire constant mean curvature graph in $\E$, and minimizes area among all entire graphs in $\E$.
	\item If the fibres of $\pi$ have infinite length, then $\Sigma$ is the unique compact minimal surface immersed in $\E$, up to vertical translations.
\end{enumerate}
\end{theorem}

\begin{proof}
Since $M$ is compact and $\int_M\frac{\tau}{\mu}=0$, Proposition~\ref{prop:existence-sections} guarantees the existence of global sections, and we can assume that the fibres of $\pi$ have infinite length without losing generality. The proof is now split in two cases depending of whether or not $M$ is topologically a sphere.

\textbf{Case I.} If $M$ is not homeomorphic to $\s^2$, then let $Q$ be a quotient of $\E$ under a vertical translation, and consider $\Sigma_0\subset Q$ to be a global smooth section. On the one hand, since $M$ is not a sphere, the compact $3$-manifold $Q$ is irreducible because its universal cover is topologically $\R^3$, which is irreducible. On the other hand, the global section $\Sigma_0$ is incompressible in $Q$ because the inclusion $\iota: \Sigma_0\to P$ induces an injective morphism $\iota_*:\pi_1(\Sigma_0)\to\pi_1(Q)$ between the homotopy groups. As a consequence of the results of Meeks, Simon and Yau~\cite{MeeksSimonYau}, there exists an embedded minimal surface $\Sigma'\subset Q$ in the isotopy class of $\Sigma_0$. 

Since $\Sigma'$ is isotopic to $\Sigma_0$, it is clear that it lifts to a compact embedded minimal surface $\Sigma\subset\E$. Since the fibres of $\pi$ have infinite length, it follows that $\Sigma$ is an entire graph by an easy application of the maximum principle. The rest of properties in the statement are now a consequence of Lemma~\ref{lemma:H=0}.

\textbf{Case II.} If $M$ is homeomorphic to $\s^2$, the technique above does not apply (note that $M\times\s^1$ is not irreducible and a global section is not incompressible). Here we will develop an \emph{ad hoc} argument based on a Calabi-type duality, and on the existence of an entire space-like graph with prescribed mean curvature $\tau$ in $M\times\R$ when the Lorentzian metric $\pi_M^*(\df s^2_M)-\mu^{-2}\pi_\R^*(\df t^2)$ is considered. Such existence follows from the work of Gerhardt~\cite{Ger}.

Parametrizing the surface given by Gerhardt as $\Phi:M\mapsto M\times\R$ with $\Phi(x)=(x,v(x))$ for some $v\in\mathcal{C}^\infty(M)$, the fact that $\Phi$ has prescribed mean curvature $\tau$ with respect to the aforesaid Lorentzian metric can be expressed as
\begin{equation}\label{prop:calabi-eqn1}
\mathrm{div}\left(\frac{\nabla v}{\mu\sqrt{\mu^2-\|\nabla v\|^2}}\right)=\frac{2\tau}{\mu},
\end{equation}
and the spacelike condition reads $\|\nabla v\|<\mu$. Given a global smooth section $F_0:M\to\E$, define $d\in\mathcal C^\infty(\E)$ as the Killing distance along fibres to $F_0$. The vector field $Z=\pi_*(\overline\nabla d)$ in $M$ satisfies $\mathrm{div}(JZ)=\frac{-2\tau}{\mu}$ by Lemma~\ref{lemma:JZ}, and therefore
\begin{equation}\label{prop:calabi-eqn2}
\mathrm{div}\left(\frac{\nabla v}{\mu\sqrt{\mu^2-\|\nabla v\|^2}}+JZ\right)=0.
\end{equation}

Since $M$ is simply connected, Poincar\'{e}'s lemma guarantees the existence of $u\in\mathcal C^\infty(M)$ such that
\begin{equation}\label{prop:calabi-eqn3}
\frac{\nabla v}{\mu\sqrt{\mu^2-\|\nabla v\|^2}}+JZ=J\nabla u.
\end{equation}
Defining $Gu=\nabla u-Z$, and applying $J$ to ~\eqref{prop:calabi-eqn3}, we get 
\begin{equation}\label{prop:calabi-eqn4}
Gu=\frac{-J\nabla v}{\mu\sqrt{\mu^2-\|\nabla v\|^2}}.
\end{equation}
Taking squared norms in~\eqref{prop:calabi-eqn4}, we reach
\begin{equation}\label{prop:calabi-eqn5}
\sqrt{\mu^2-\|\nabla v\|^2}=\frac{1}{\sqrt{\mu^{-2}+\|Gu\|^2}}.
\end{equation}
Now plugging~\eqref{prop:calabi-eqn4} into~\eqref{prop:calabi-eqn5}, leads to 
\begin{equation}\label{prop:calabi-eqn6}
\frac{\mu\cdot Gu}{\sqrt{\mu^{-2}+\|Gu\|^2}}=-J\nabla v,
\end{equation}
and taking the divergence in~\eqref{prop:calabi-eqn6} it follows that the graph of $u$ with respect to $F_0$ is an entire minimal graph (note that $\mathrm{div}(J\nabla v)=0$). The remaining properties in the statement follow from Lemma~\ref{lemma:H=0} and the maximum principle.
\end{proof}

\begin{remark}\label{rmk:complete-multigraphs}
If $M$ is compact but not simply connected, Gerhardt's result still holds. Passing to the universal Riemannian covering surface of $M$, an argument similar to that of Case II in the proof of Theorem~\ref{thm:existence-entire-minimal} can be applied giving rise to a complete minimal surface in $\E$ everywhere transversal to the Killing direction (a so-called \emph{vertical multigraph}), but it may fail to be an entire graph because of a period problem.
\end{remark}

Killing submersions with the topology of the Hopf projection also admit minimal spheres. This follows from a result of Simon that $\s^3$ endowed with any Riemannian metric admits an embedded minimal sphere~\cite{Simon}. These are the only two scenarios where minimal spheres exist.

\begin{theorem}\label{thm:minimal-spheres}
Let $\pi:\E\to M$ be a Killing submersion. Then $\E$ admits a immersed minimal sphere $\Sigma$ if and only if $M$ is topologically $\s^2$.
\end{theorem} 

\begin{proof}
Since spheres are simply connected, there exists an immersed minimal sphere $\Sigma\subset\E$ if and only if there exists an immersed minimal sphere $\widetilde\Sigma\subset\widetilde\E$, being $\widetilde\pi:\widetilde\E\to\widetilde M$ the universal Riemannian covering given by Theorem~\ref{thm:classification-other}. If the fibres of $\widetilde\pi$ have infinite length, then any immersed minimal sphere $\widetilde\Sigma\subset\widetilde\E$ must be an entire graph by the maximum principle. Thus either $\widetilde\Sigma$ is an entire graph or the fibres of $\pi$ have finite length. 

In the first case, $\widetilde\pi$ induces an homeomorphism from $\widetilde\Sigma$ to $\widetilde M$, so $\widetilde M$ is topologically $\s^2$ and the submersion is trivial. As $\widetilde\Sigma$ is an entire minimal graph and $\widetilde\E$ is homeomorphic to $\s^2\times\R$, we conclude that $\widetilde\Sigma$ corresponds to one of the examples in Theorem~\ref{thm:existence-entire-minimal}. In the second case, $\widetilde\pi$ is the Hopf fibration by Theorem~\ref{thm:classification-sphere} and the existence is a consequence of Simon's result~\cite{Simon}. Hence both cases lead to $\widetilde M$ being topologically $\s^2$. Theorem~\ref{thm:classification-other} implies that $\E$ is a quotient of $\widetilde\E$ under a group $G$ of Killing isometries, and the only orientable quotient of a sphere is the sphere itself, so we deduce that $G$ projects to the trivial group, i.e., $G$ only consists of vertical translations. Therefore $\pi:\E\to M$ is again a Killing submersion over $\s^2$, and the statement follows (see also Remark~\ref{rmk:classification-sphere-general}).
\end{proof}

\section{Compact stable surfaces with constant mean curvature}\label{sec:stability}

We begin by introducing the stability operator $L=\Delta+|A|^2+\mathrm{Ric}(N)$ of a constant mean curvature surface immersed in the total space of a Killing submersion, where $A$ is the shape operator of the immersion and $N$ is a unit normal vector field. Next lemma gives another expression for $L$ in terms of the scalar curvature.

\begin{lemma}\label{lemma:L}
Let $\pi:\E\to M$ be a Killing submersion with bundle curvature $\tau$ and Killing length $\mu$. The stability operator of a constant mean curvature $H$ surface $\Sigma$ immersed in $\E$ reads
\[L=\Delta-K+4H^2+S-\det(A),\]
where $K$ denotes the Gaussian curvature, respectively, and the scalar curvature $S$ of $\E$ (as a function on $M$) satisfies
\[\frac{1}{2}S=K_M-\tau^2-\frac{\Delta\mu}{\mu},\] 
being $K_M$ the Gaussian curvature of $M$ and $\Delta$ its Laplace operator.
\end{lemma}

\begin{proof}
The computation is local so we shall consider the model metric on $\E$ given by~\eqref{eqn:metric} and the metric in the base $\lambda^2(\df x^2+\df y^2)$. Using the expression of the Lie bracket~\eqref{eqn:bracket} and Koszul formula, we derive the Levi-Civita connection in $\E$ applied to the frame given by~\eqref{eqn:frame}
\begin{equation}\label{eqn:levi-civita}
\begin{aligned}
\overline\nabla_{E_1}E_1&=-\frac{\lambda_y}{\lambda^2}E_2,&
\overline\nabla_{E_1}E_2&=\frac{\lambda_y}{\lambda^2}E_1+\tau E_3,&
\overline\nabla_{E_1}E_3&=-\tau E_2,\\
\overline\nabla_{E_2}E_1&=\frac{\lambda_x}{\lambda^2}E_2-\tau E_3,&
\overline\nabla_{E_2}E_2&=-\frac{\lambda_x}{\lambda^2}E_1,&
\overline\nabla_{E_2}E_3&=\tau E_1,\\
\overline\nabla_{E_3}E_1&=-\tau E_2+\frac{\mu_x}{\lambda\mu}E_3,&
\overline\nabla_{E_3}E_2&=\tau E_1+\frac{\mu_y}{\lambda\mu}E_3,&
\overline\nabla_{E_3}E_3&=-\frac{1}{\mu}\overline\nabla\mu.
\end{aligned}
\end{equation}
From~\eqref{eqn:levi-civita} we get the following sectional curvatures of the basic planes spanned by $\{E_1,E_2,E_3\}$ ($R$ denotes the Riemannian curvature tensor):
\begin{equation}\label{lemma:L:eqn1}
\begin{aligned}
R(E_1,E_2,E_2,E_1)&=K_M-3\tau^2,\\
R(E_1,E_3,E_3,E_1)&=\tau^2-\tfrac{1}{\mu}E_1(E_1(\mu))-\tfrac{1}{\lambda\mu}E_2(\lambda)E_2(\mu),\\
R(E_2,E_3,E_3,E_2)&=\tau^2-\tfrac{1}{\mu}E_2(E_2(\mu))-\tfrac{1}{\lambda\mu}E_1(\lambda)E_1(\mu).
\end{aligned}
\end{equation}
The scalar curvature $S$ is twice the result of adding up the three quantities in~\eqref{lemma:L:eqn1}, which leads to the expression in the statement. Finally the given expression for $L$ holds in any $3$-manifold (see~\cite[Section~10]{MeeksPerezRos}).
\end{proof}

Note that the \emph{angle function} of the immersion, defined as $\nu=\langle N,\xi\rangle$, where $\xi$ is the vertical Killing vector field, lies in the kernel of the Jacobi operator~\cite[Proposition 2.12]{BCE}. Assuming that $\nu$ is identically zero or never vanishes gives rise to two distinguished families of surfaces in $\E$:
\begin{itemize}
 \item If $\nu\equiv 0$, then $\Sigma$ is everywhere vertical, so there exists a curve $\Gamma\subset M$ such that $\Sigma=\pi^{-1}(\Gamma)$ and $\Sigma$ is called the \emph{vertical cylinder} over $\Gamma$.

 Parametrizing $\Gamma$ as $\gamma:[a,b]\to M$ with unit-speed, and taking a horizontal lift $T$ of $\gamma'$ along $\Sigma$, then $\{T,E_3=\frac{1}{\mu}\xi\}$ is a orthonormal frame in $\Sigma$. Using \eqref{eqn:levi-civita} we compute the second fundamental form in this frame as
 \[\sigma\equiv\left(\begin{matrix}
  \langle\overline\nabla_TT,N\rangle&\langle\overline\nabla_TE_3,N\rangle\\
  \langle\overline\nabla_{E_3}T,N\rangle&\langle\overline\nabla_{E_3}E_3,N\rangle
 \end{matrix} \right)=\left(\begin{matrix}\kappa_g&\tau\\ \tau&-\eta(\log(\mu))\end{matrix}\right)\]
 where $\kappa_g$ is the geodesic curvature of $\gamma$ in $M$ and $\eta=\pi_*N$. Hence $\Sigma$ has mean curvature $H$ if and only if $\kappa_g=2H+\eta(\log(\mu))$. Using the classical theory of \textsc{ODE}s, if $M$ is complete, then for any point $x\in M$ and any direction $v\in T_xM$ there exists a unique constant speed curve $\gamma:\R\to M$ with $\gamma(0)=x$ and $\gamma'(0)=v$, and such that $\pi^{-1}(\Gamma)$ has constant mean curvature $H$.

 Using the parametrization $(t,s)\mapsto\phi_t(\widetilde\gamma(s))$, where $\{\phi_t\}$ is the group of vertical translations and $\widetilde\gamma$ is a horizontal lift of $\gamma$, the metric in $\Sigma$ is expressed as $\mu(\gamma(s))^2\df t^2+\df s^2$, which is not flat in general as it is in the unitary case (see~\cite[Proposition 2.12]{EO})

 \item If $\nu$ has no zeroes, then $\Sigma$ is everywhere transversal to the Killing vector field, and it is called a \emph{vertical multigraph}. Note that $\Sigma$ is a graph if and only if $\pi_{|\Sigma}:\Sigma\to M$ is injective.

 A vertical multigraph with constant mean curvature is always stable since $\nu$ is a positive function such that $L\nu=0$.
\end{itemize}

\begin{theorem}\label{thm:compact-stable}
Let $\pi:\E\to M$ be a Killing submersion with bundle curvature $\tau$ and Killing length $\mu$, and let $\Sigma$ be a compact orientable stable surface with constant mean curvature immersed in $\E$. Then one of the following assertions holds:
\begin{enumerate}[label=(\roman*)]
 \item $\Sigma$ is an entire minimal graph, $M$ is compact and $\int_M\frac{\tau}{\mu}=0$.
 \item $\Sigma$ is a vertical cylinder over a closed curve in $M$ and the fibres of $\pi$ are compact. Hence $\Sigma$ is topologically a torus.
\end{enumerate}
In particular, if both $M$ and the fibres of $\pi$ are not compact, then $\E$ does not admit compact orientable stable surfaces with constant mean curvature.
\end{theorem}

\begin{proof}
Stability implies that the angle function $\nu$, which is bounded on $\Sigma$ and lies in the kernel of the stability operator, is identically zero or it never vanishes. If $\nu\equiv 0$, then $\Sigma=\pi^{-1}(\Gamma)$ for a certain curve $\Gamma\subseteq M$. Since $\Sigma$ is compact, the fibres of $\pi$ must have finite length and we get to (ii). If $\nu$ never vanishes, then (i) follows from Lemma~\ref{lemma:H=0}.
\end{proof}

On the one hand, if $M$ is compact and $\int_M\frac{\tau}{\mu}=0$, then there are always surfaces satisfying item (i) by Theorem~\ref{thm:existence-entire-minimal}. On the other hand, compact vertical cylinders are topologically tori, but a characterization of their stability in terms of the corresponding curve in the base seems to be a tough task. It is important to mention that they may exist; for instance, if $\Gamma$ is a curve with constant geodesic curvature $\kappa_g\leq 1$ in the hyperbolic plane $\h^2$, then $\Gamma\times\s^1$ is a stable torus with constant mean curvature $H=\frac{1}{2}\kappa_g$ in the Riemannian product space $\h^2\times\s^1$.

\begin{corollary}\label{coro:stable-spheres}
Let $\pi:\E\to M$ be a Killing submersion and let $\Sigma$ be a compact stable surface with constant mean curvature immersed in $\E$. If $\Sigma$ is not topologically a torus, then it is an entire minimal graph.
\end{corollary}

Theorem~\ref{thm:compact-stable} also allows us to give an alternative proof to the classification of compact stable minimal surfaces in homogeneous $3$-manifolds given by Meeks and Pérez~\cite[Theorem~4.18]{MeeksPerez}.

\begin{corollary}
Let $X$ be a simply connected homogeneous Riemannian $3$-manifold. If there exists a compact orientable stable surface $\Sigma$ with constant mean curvature immersed in $X$, then $X$ is isometric to $\s^2\times\R$, being $\Sigma$ a horizontal slice $\s^2\times\{t_0\}$.
\end{corollary}

\begin{proof}
If $X$ is homeomorphic to $\R^3$, then Example~\ref{ex:homogeneous-R3} shows that $X$ is the total space of a Killing submersion with non-compact base or fibres. If $X$ is homeomorphic to $\s^3$, then there exists a Hopf-like Killing submersion $\pi:X\to\s^2$ for which $\Sigma$ is not vertical by Example~\ref{ex:homogeneous-R3}, so $\Sigma$ should be an entire graph for $\pi$ but the Hopf fibration does not admit entire sections. We conclude that $X$ isometric $\s^2\times\R$ (see~\cite[Theorem 2.4]{MeeksPerez}), and $\Sigma$ is an entire minimal graph by Theorem~\ref{thm:compact-stable}, unique up to vertical translations, so $\Sigma$ must be a slice $\s^2\times\{t_0\}$.
\end{proof}

Finally we will give a last characterization of our entire minimal graphs dropping the assumption of compactness. We will apply a result of Meeks, Pérez and Ros~\cite[Theorem~2.13]{MeeksPerezRos} based on a previous result of Rosenberg~\cite{Rosenberg}.

\begin{corollary}\label{coro:complete-stable}
Let $\pi:\E\to M$ be a Killing submersion and suppose that $\Sigma$ is a complete stable surface with constant mean curvature $H$ immersed in $\E$.
\begin{enumerate}
 \item If $3H^2>\sup_{\pi(\Sigma)}\{\tau^2-K_M+\frac{1}{\mu}\Delta u\}$, then $\pi$ is topologically the projection $\s^2\times\R\to\s^2$ and $\Sigma$ is an entire minimal graph.

 \item If $3H^2=\sup_{\pi(\Sigma)}\{\tau^2-K_M+\frac{1}{\mu}\Delta u\}$ and $\mu$ is bounded, then either $\Sigma$ is a vertical multigraph or a vertical cylinder.
\end{enumerate}
\end{corollary}

\begin{proof}
The condition in item (1) is equivalent to $3H^2+\frac{1}{2}S\geq c>0$ for some $c>0$ by Lemma~\ref{lemma:L}, so \cite[Theorem~2.13]{MeeksPerezRos} guarantees that $\Sigma$ is homeomorphic to $\mathbb{S}^2$ or to the real projective plane $\mathbb{RP}^2$. Taking the universal covering of $\Sigma$ if necessary, both cases lead to the existence of a constant mean curvature sphere $\widetilde\Sigma$ immersed in $\E$, which is also stable (see the comments below Assertion 2.2 in \cite{MeeksPerezRos}). The result now follows from Corollary~\ref{coro:stable-spheres}.

If the condition in item (2) holds, \cite[Theorem~2.13]{MeeksPerezRos} says that $\Sigma$ has at most quadratic area growth, so it is in particular parabolic. Hence~\cite[Corollary 1]{MPR} implies that the angle function of $\Sigma$, which is bounded because $\mu$ is assumed bounded, is identically zero or never vanishes. This dicotomy reflects the two possible scenarios depicted in the statement.
\end{proof}

The hypothesis in the statement of Corollary~\ref{coro:complete-stable} are not expected to be optimal, as shown in~\cite[Corollary~3]{MPR} for $\E(\kappa,\tau)$-spaces. Also $\sup_{\pi(\Sigma)}\{\tau^2-K_M+\frac{1}{\mu}\Delta u\}$ might be infinite, in that case no information about the surface is obtained.

\end{document}